\documentclass[11pt, reqno]{amsart}

\usepackage{amsmath, amscd, amssymb}
\usepackage{amsthm, amsfonts}
\usepackage{latexsym}
\usepackage[T1]{fontenc}
\usepackage{tikz}
\usepackage[shortlabels]{enumitem}
\usepackage[reset,text={35pc,8.6in},centering]{geometry} 

\allowdisplaybreaks[3]

\theoremstyle{plain}
\newtheorem{lemma}{Lemma}[section]
\newtheorem{prop}[lemma]{Proposition}

\newtheorem{cor}[lemma]{Corollary}
\newtheorem*{MainTheorem}{MAIN THEOREM}

\theoremstyle{definition}

\newtheorem{example}[lemma]{Example}
\newtheorem{remark}[lemma]{Remark}
\numberwithin{equation}{section} 

\newcommand{\C}{{\mathbb C}}

\newcommand{\Z}{{\mathbb Z}}
\newcommand{\ve}{\varepsilon}

\newcommand{\GL}{\mathrm{GL}}
\newcommand{\Gln}{{\GL_n}}
\newcommand{\Glk}{\GL_k}
\newcommand{\rM}{{\mathrm M}}
\newcommand{\calp}{\mathcal{P}}
\newcommand{\calq}{\mathcal{Q}}
\newcommand{\calr}{\mathcal{R}}
\newcommand{\Mnk}{{\rM_{n,k}}}
\newcommand{\pmnk}{{\mathcal{P}(\Mnk)}}
\newcommand{\bfm}{\mathbf{m}}

\newcommand{\bbf}{\mathbf{B}}

\DeclareMathOperator{\sgn}{sgn}

\DeclareMathOperator{\Gr}{Gr}
\DeclareMathOperator{\wt}{wt}

\newcommand{\LM}{\mathrm{LM}}
\newcommand{\SST}{\mathrm{SST}}

\newcommand{\len}[1]{\ell(#1)}
\newlength{\HeightOfFigure}
\newcommand{\vcentralize}[1]{
\settoheight{\HeightOfFigure}{#1}
\divide\HeightOfFigure by 2
\addtolength{\HeightOfFigure}{-2.5pt}
\lower\HeightOfFigure\hbox{#1}
}
\newcommand{\set}[3][undefined]{\expandafter\ifx\csname#1\endcsname\undefined \left\{#2\,:\,#3\right\}%
\else\csname#1l\endcsname\{#2\,:\,#3\csname#1r\endcsname\}\fi}
\DeclareMathOperator{\diag}{diag}
\DeclareMathOperator{\Span}{Span}
\DeclareMathOperator{\Hom}{Hom}
\newcommand{\sym}[1]{\mathfrak{S}_{#1}} 
\newcommand{\B}{\mathbf{B}}
\newcommand{\Binv}{\B_{\text{\upshape inv}}}
\newcommand{\Bsgn}{\B_{\text{\upshape sgn}}}
\newcommand{\Sinv}{{\mathbf{S}}_{\text{\upshape inv}}}
\newcommand{\Ssgn}{{\mathbf{S}}_{\text{\upshape sgn}}}
\newcommand{\cQ}{\mathcal{Q}}
\newcommand{\cQinv}[1][k]{\cQ^{\sym#1}}
\newcommand{\cQsgn}{\cQ^{\text{\upshape sgn}}}
\newcommand{\triv}{\mathbf{1}} 
\newcommand{\symirrep}[2]{\mathcal{S}_{#1}^{#2}} 
\newcommand{\floor}[1]{\left\lfloor#1\right\rfloor}

\newcommand{\om}[1]{\widehat{#1}} 

\newcommand{\eqed}{\pushQED{\qed}\qedhere\popQED}
\newcommand{\eqspace}{\phantom{{}={}}}

\begin{document}

\title[Highest weight vectors in plethysms]
{Highest weight vectors in plethysms}

\author{Kazufumi Kimoto}
\address{Department of Mathematical Sciences,
University of the Ryukyus,
1 Senbaru, Nishihara, Okinawa 903-0213, JAPAN}
\email{kimoto@math.u-ryukyu.ac.jp}

\author{Soo Teck Lee}
\address{Department of Mathematics, National University of Singapore,
Block S17, 10 Lower Kent Ridge Road, Singapore 119076, Republic of Singapore}
\email{matleest@nus.edu.sg}

\thanks{{
The first named author is partially supported by Grant-in-Aid for Scientific Research (C) No.25400044, JSPS and by JST CREST Grant Number JPMJCR14D6, Japan.
The second named author is supported by NUS grant R-146-000-252-114.
}}

\begin{abstract}
We realize the $\Gln(\C)$-modules $S^k(S^m(\C^n))$ and $\Lambda^k(S^m(\C^n))$
as spaces of polynomial functions on $n\times k$ matrices.
In the case $k=3$, we describe explicitly all the $\Gln(\C)$-highest weight vectors
which occur in $S^3(S^m(\C^n))$ and in $\Lambda^3(S^m(\C^n))$ respectively.
\end{abstract}


\maketitle

\section{Introduction}

Let $V_1, V_2$ and $V_3$ be finite dimensional complex vector spaces,
and let $\pi_1\colon\GL(V_1)\to\GL(V_2)$ and $\pi_2\colon\GL(V_2)\to\GL(V_3)$
be polynomial representations of $\GL(V_1)$ and $\GL(V_2)$ respectively.
Then the composition $\pi_2\circ\pi_1\colon\GL(V_1)\to\GL(V_3)$ is also a polynomial representation of $\GL(V_1)$.
If $\chi_1$ and $\chi_2$ are the characters of $\pi_1$ and $\pi_2$ respectively,
then the character of $\pi_2\circ\pi_1$ is called the \emph{plethysm} of $\chi_1$ and $\chi_2$
and is denoted by $\chi_2\circ \chi_1$ (or $\chi_2[\chi_1]$).
One of the main open problems in combinatorics is to express the plethysm $s_\lambda\circ s_\mu$ of two Schur polynomials,
which are characters of irreducible polynomial representations of the general linear groups
with highest weights $\lambda$ and $\mu$ respectively,
as a linear combination of Schur polynomials.

\medskip
Let us look at the case of the complete symmetric polynomials,
i.e. the plethysms of the form $h_k\circ h_m$.
The problem in this case is equivalent to determining the irreducible decomposition of the $\GL_n$-module $S^k(S^m(\C^n))$.
Several explicit results are known,
for instance,
\begin{equation*}
h_k\circ h_2=\sum_{\mu}s_\mu
\end{equation*}
for an arbitrary positive integer $k$, where $\mu$ runs through all the even partitions of $2k$.
The formula is actually equivalent to the identity
\begin{equation*}
\sum_{\text{$\mu$:even}}s_\mu=\prod_{i\le j}(1-x_ix_j)^{-1}
\end{equation*}
due to Littlewood \cite{L1950}.
There are also results on $h_2\circ h_m$, $h_3\circ h_m$ and $h_4\circ h_m$
for an arbitrary positive integer $m$:
the first and second case are due to Thrall \cite{T1942} (see also \cite{P1972}),
and the third case is due to Foulkes \cite{Foulkes1954}.
On the other hand, using representation theory and in particular \emph{$(\GL_n,\GL_m)$-duality}, 
Howe \cite{H1} describes the multiplicities in $S^k(S^m(\C^n))$ for $k\le 4$.
An example of recent developments of plethysms is \cite{dBPW}, which studies certain stability conditions on the coefficients in the expansion with respect to Schur polynomials by using combinatorics on tableaux whose entries are also tableaux.

\medskip

In this paper, we study plethysms using representation theory and an approach inspired by \cite{H1}.
We realize the $\Gln$-modules $S^k(S^m(\C^n))$ and $\Lambda^k(S^m(\C^n))$ as spaces of polynomial functions.
In the case when $k=3$, we obtain all the $\Gln$-highest weight vectors
which occur in $S^3(S^m(\C^n))$ and in $\Lambda^3(S^m(\C^n))$ respectively.
Our results on the explicit highest weight vectors in of $S^3(S^m(\C^n))$ and $\Lambda^3(S^m(\C^n))$
give more refined information than multiplicities on the structure of these modules. 

\medskip

We now describe our approach and results in more details. 
Assume that $n\ge k$ and let $\sym k$ denote the symmetric group on $[k]=\{1,2,\dots,k\}$. Following Howe \cite{H1}, 
we consider an algebra $\calq$ of polynomial functions on the space  of all $n\times k$ complex matrices
with the properties that $\Gln\times\sym k$ acts on $\calq$ by algebra homomorphisms, and $\calq$ has a decomposition
\[
\calq\cong\bigoplus_{m=0}^\infty \calq_m
\]
where for each $m$, $\calq_m\cong S^m(\C^n)^{\otimes k}$ as a $\Gln$-module,
and $\sym k$ acts on $\calq_m$ by permuting the factors.
Next, we define
\begin{align*}
\calq^{\sym k}&=\set{f\in\calq}{\tau.f=f,\; \forall \tau\in \sym k}, \\
\calq^{\sgn}&=\set{f\in\calq}{\tau.f=(\sgn\tau)f,\; \forall \tau\in \sym k}.
\end{align*}
Then $\calq^{\sym k}$ and $\calq^{\sgn}$ have decompositions
\[
\calq^{\sym k}
=\bigoplus_{m=0}^\infty \calq^{\sym k}_m,\quad \calq^{\sgn}
=\bigoplus_{m=1}^\infty \calq^{\sgn}_m
\]
where for each $m\ge 0$,
\[
\calq^{\sym k}_m \cong \bigl(S^m(\C^n)^{\otimes k}\bigr)^{\sym k}\cong S^k(S^m(\C^n)),
\]
and for each $m\ge 1$,
\[
\calq^{\sgn}_m\cong \bigl(S^m(\C^n)^{\otimes k}\bigr)^{\sgn} \cong \Lambda^k(S^m(\C^n)).
\]
In this way, we obtain explicit models for  $S^k(S^m(\C^n))$ and $\Lambda^k(S^m(\C^n))$. 

\medskip

Our goal is to determine the $\Gln$-highest weight vectors
in the spaces $S^k(S^m(\C^n))$ and $\Lambda^k(S^m(\C^n))$.
Let $U_n$ be the standard maximal unipotent subgroup of $\Gln$,
and consider the algebra $(\cQinv)^{U_n}$
and the space $(\cQsgn)^{U_n}$ of $U_n$-invariants in $\cQinv$ and $\cQsgn$ respectively. 
If suitable bases of $(\cQinv_m)^{U_n}$ and $(\cQsgn_m)^{U_n}$ are known,
then they can be used to describe all the $\Gln$-highest weight vectors
which occur in $S^k(S^m(\C^n))$ and in $\Lambda^k(S^m(\C^n))$.

\medskip

We shall implement the above procedure for the case $k=3$.
In this case, the algebra $\calq$ is a subalgebra of the polynomial algebra
on the variables $x_{ij}$, $1\le i\le n,\ 1\le j\le 3$.
For $1\le i,j\le 3$, let 
\[
\delta_{ij}=
\begin{vmatrix}
x_{1i}&x_{1j}\\
x_{2i}&x_{2j}
\end{vmatrix}.
\]
Then we show that the algebra $\calq^{U_n}$ of $U_n$-invariants in $\calq$ is generated by
\[
\alpha_1=x_{11}x_{12}x_{13},\quad \beta_2=x_{13}\delta_{12},
\quad \beta_3=x_{12}\delta_{13},
\quad\gamma_1=
\begin{vmatrix}
x_{11}&x_{12}&x_{13}\\
x_{21}&x_{22}&x_{23}\\
x_{31}&x_{32}&x_{33}
\end{vmatrix},\quad
\gamma_2=\delta_{12}\delta_{13}\delta_{23}.
\] 

\begin{MainTheorem}
Let
\[
\alpha_2=\beta^2_2+\beta^2_3-\beta_2\beta_3,\quad
\alpha_3=2(\beta^3_2+\beta^3_3)-3(\beta^2_2\beta_3+\beta_2\beta^2_3).
\]
\begin{enumerate}[(i)]
\item The set
\[
\Binv
=\set[big]{\alpha^a_1\alpha^b_2\alpha^c_3\gamma_1^{2d+f}\gamma_2^{2e+f}}{a,b,d,e\in\Z_{\ge 0}, c,f\in\{0,1\}}
\]
is a basis for $(\calq^{\sym3})^{U_n}$.
\item The set  
\begin{align*}
\Bsgn&=\set[big]{\alpha^a_1\alpha^b_2\alpha^c_3\gamma_1^{2d+1}\gamma_2^{2e}}{a,b,d,e\in\Z_{\ge 0}, c\in\{0,1\}}\\
&\qquad{}\cup
\set[big]{\alpha^a_1\alpha^b_2\alpha^c_3\gamma_1^{2d}\gamma_2^{2e+1}}{a,b,d,e\in\Z_{\ge 0}, c\in\{0,1\}}
\end{align*}
is a basis for $(\calq^{\sgn})^{U_n}$.
\end{enumerate}
\end{MainTheorem}

Let $A_n$ be the diagonal torus of $\Gln$.
For $\lambda=(\lambda_1,\dots,\lambda_n)\in\Z^n$,
let $\psi^\lambda_n\colon A_n\to\C$ be defined by
\begin{equation}\label{eq:psilambda}
\psi^\lambda_n(t)=t_1^{\lambda_1}\cdots t_n^{\lambda_n},\quad\forall t=\diag(t_1,\dots,t_n)\in A_n,
\end{equation}
where $\diag(t_1,\dots,t_n)$ denotes the $n\times n$ diagonal matrix with diagonal entries $t_1,\dots,t_n$.
If $f\in \calq_m$ and $t.f=\psi^\lambda(t)f$, then we write
\begin{equation}\label{eq:grwt}
\Gr(f)=m,\quad \text{and}\quad \wt(f)=\lambda.
\end{equation}
For a Young diagram $D$, we shall denote the number of rows in $D$ by $\ell(D)$ and the number of boxes in $D$ by $|D|$.
It is well known that the irreducible polynomial representations of $\Gln$ are labeled
by the set of all Young diagrams $D$ with $\ell(D)\le n$ (\cite{F,GW,H2}).
Specifically, for such a Young diagram $D$, we identify it with the element $(\lambda_1,\dots,\lambda_n)$ of $\Z^n$
where for each $i$, $\lambda_i$ is the number of boxes in the $i$-th row of $D$,
and denote the irreducible representation of $\Gln$ with highest weight $\psi^D_n$ by $\rho^D_n$.
It follows from the main theorem that the set
\[
\Binv(m,D)=\set{f\in\Binv}{\Gr(f)=m,\ \wt(f)=D}
\]
is a basis for the space of highest weight vectors of weight $\psi^D_n$ in $S^3(S^m(\C^n))$,
and the set 
\[
\Bsgn(m,D)=\set{f\in\Bsgn}{\Gr(f)=m,\ \wt(f)=D}
\]
is a basis for the space of highest weight vectors of weight $\psi^D_n$ in $\Lambda^3(S^m(\C^n))$.
In particular, the cardinality of $\Binv(m,D)$ and the cardinality of $\Binv(m,D)$
give the multiplicity of $\rho^D_n$ in $S^3(S^m(\C^n))$ and in $\Lambda^3(S^m(\C^n))$ respectively.

\subsection*{General convention}

Let $G$ be a group.
We denote by $\triv=\triv_G$ the trivial representation of $G$.
For a $G$-module $V$, we denote by $V^G$ the subspace consisting of all $G$-invariants in $V$.
Namely,
\begin{equation*}
V^G:=\set{v\in V}{g.v=v,\;\forall g\in G}.
\end{equation*}
When $V$ is also an $\sym k$-module, we put
\begin{equation*}
V^{\sgn}:=\set{v\in V}{\tau.v=(\sgn\tau)v,\; \forall\tau\in\sym k},
\end{equation*}
where $\sgn\tau$ is the sign of $\tau\in\sym k$.
Notice that if the actions of $G$ and $\sym k$ commute with each other,
then $V^G$ is an $\sym k$-submodule, $V^{\sym k}$ and $V^{\sgn}$ are $G$-submodules, and
\begin{equation*}
(V^G)^{\sym k}=(V^{\sym k})^G,\qquad
(V^G)^{\sgn}=(V^{\sgn})^G.
\end{equation*}

\section{Spaces of $U_n$-invariants}\label{sec:uninv}

In this section, we realize the $\Gln$-modules $S^k(S^m(\C^n))$ and $\Lambda^k(S^m(\C^n))$
as spaces of polynomial functions and study the highest weight vectors which occur in these spaces.

\subsection{The algebras $\calp$ and $\calq$} 

This subsection is based on \cite{H1}. 
Let $\calp=\pmnk$ be the algebra of polynomial functions on the space $\Mnk=\Mnk(\C)$ of all $n\times k$ complex matrices,
and let $\Gln \times \Glk$ act on $\calp$ by
\begin{equation} \label{glnk}
[(g,h) \cdot f](X)=f(g^tXh)
\end{equation}
where $(g,h) \in \Gln \times \Glk$, $f \in \pmnk$, and $X \in \Mnk$.
We restrict this action to the subgroup $\Gln\times A_k$ of $\Gln\times\Glk$.
Then by a standard argument using $(\Gln,\Glk)$-duality (\cite{H2}),
we see that $\calp$ has a decomposition into a direct sum of $\Gln\times A_k$-submodules
\begin{equation}\label{eq:pdecom}
\calp=\bigoplus_{\bfm\in\Z^k_{\ge 0}}\calp_{\bfm}
\end{equation}
where $\Z_{\ge 0}$ denotes the set of all nonnegative integers, and for each $\bfm=(m_1,\dots,m_k)\in\Z^k_{\ge 0}$, 
\[
\calp_{\bfm}\cong S^{m_1}(\C^n)\otimes\cdots\otimes S^{m_k}(\C^n)
\]
as a $\Gln$-module and $A_k$ acts on $\calp_{\bfm}$ by the character $\psi^\bfm_k$, i.e.
\[
\calp_\bfm=\set{f\in\calp}{t.f=\psi^{\bfm}_k(t)f,\, \forall t\in A_k}.
\]
Thus equation \eqref{eq:pdecom} defines a graded algebra structure on $\calp$
graded by the semigroup $\Z^k_{\ge 0}$.

\medskip  
For each $m\in\Z_{\ge 0}$, let
\[
(m^k)=(\overbrace{m,m,\dots,m}^k).
\]
Let $\calq$ be the direct sum of all the homogeneous components of $\calp$ of the form $\calp_{(m^k)}$, i.e.
\begin{equation}\label{eq:qdecom}
\calq=\bigoplus_{m\in\Z_{\ge 0}} \calq_m
\end{equation}
where for each $m\in\Z_{\ge 0}$, 
\[
\calq_m=\calp_{(m^k)}\cong S^m(\C^n)^{\otimes k}
=\overbrace{S^m(\C^n)\otimes\cdots\otimes S^m(\C^n)}^k.
\] 
Clearly $\calq$ is a graded subalgebra of $\calp$.
We also note that the symmetric group $\sym k$ acts on each $\calq_m$ by permuting its factors,
and by taking direct sum, we obtain an action of $\sym k$ on the algebra $\calq$ by algebra automorphisms.
In fact, if we identify $\sym k$ with the group of permutation matrices in $\Glk$,
then this action by $\sym k$ coincides with the action obtained by restriction of the action by $\Glk$ defined in equation \eqref{glnk}.

\subsection{The algebra $\calq^{\sym k}$ and the space $\calq^{\sgn}$}

The subalgebra $\calq^{\sym k}$ of $\calq$ consisting of all the $\sym k$-invariants in $\calq$ has a decomposition
\begin{equation}\label{eq:qskdecom}
\calq^{\sym k}=\bigoplus_{m=0}^\infty \calq^{\sym k}_m
\end{equation}
where for each $m\ge 0$,
\[
\calq^{\sym k}_m=\calq^{\sym k}\cap \calq_m\cong\bigl(S^m(\C^n)^{\otimes k}\bigr)^{\sym k}\cong S^k(S^m(\C^n)).
\]
On the other hand, the $\sgn$-isotypic component $\calq^{\sgn}$ of $\calq$ has a decomposition
\begin{equation}\label{eq:qsgndecom}\calq^{\sgn}=\bigoplus_{m=1}^\infty \calq^{\sgn}_m
\end{equation}
where for each $m\ge 1$,
\[
\calq^{\sgn}_m=\calq^{\sgn}\cap \calq_m\cong\bigl(S^m(\C^n)^{\otimes k}\bigr)^{\sgn}\cong \Lambda^k(S^m(\C^n)).
\]
In this way, we obtain explicit models for  $S^k(S^m(\C^n))$ and $\Lambda^k(S^m(\C^n))$. 
Note that $\calq^{\sgn}$ is not a subalgebra of $\calq$.

\subsection{$U_n$-invariants}\label{subsec:uninv}

Let $U_n$ be the subgroup of $\Gln$ consisting of all upper triangular matrices such that all diagonal entries are $1$.
Then $U_n$ is the standard maximal unipotent subgroup of $\Gln$.
A nonzero element of $\calq$ is a $\Gln$-highest weight vector if it is fixed by all elements of $U_n$
and it is an eigenvector for each element of $A_n$.
Thus, in order to determine the $\Gln$-module structure of $\calq^{\sym k}$ and $\calq^{\sgn}$,
we consider the spaces
$\calq^{U_n}$, $(\calq^{\sym k})^{U_n}$ and $(\calq^{\sgn})^{U_n}$
of $U_n$-invariants in $\calq$, $\calq^{\sym k}$ and $\calq^{\sgn}$ respectively.
By equations \eqref{eq:qskdecom} and \eqref{eq:qsgndecom}, we have
\begin{equation}\label{eq:qskundecom}
\bigl(\calq^{\sym k}\bigr)^{U_n}=\bigoplus_{m=0}^\infty \bigl(\calq^{\sym k}_m\bigr)^{U_n}
\end{equation}
where for each $m\ge 0$,
\[
\bigl(\calq^{\sym k}_m\bigr)^{U_n}=\calq^{\sym k}_m\cap\calq^{U_n} \cong S^k(S^m(\C^n))^{U_n},
\]
and 
\begin{equation}\label{eq:qsgnundecom}
\bigl(\calq^{\sgn}\bigr)^{U_n}=\bigoplus_{m=1}^\infty \bigl(\calq^{\sgn}_m\bigr)^{U_n}
\end{equation}
where for each $m\ge 1$,
\[
\bigl(\calq^{\sgn}_m\bigr)^{U_n}=\calq^{\sgn}_m\cap\calq^{U_n} \cong \Lambda^k(S^m(\C^n))^{U_n}.
\]
Hence, the structure of the algebra $(\calq^{\sym k}_m)^{U_n}$
and the space $(\calq^{\sgn}_m)^{U_n}$ encode information
on the $\Gln$-module structure of $S^k(S^m(\C^n))$ and $\Lambda^k(S^m(\C^n))$.
In particular, if bases of
$(\calq^{\sym k}_m)^{U_n}$ and $(\calq^{\sgn}_m)^{U_n}$
which are compatible with the decompositions \eqref{eq:qskundecom} and \eqref{eq:qsgnundecom} are known,
then they can be used to describe all the $\Gln$-highest weight vectors which occur in $S^k(S^m(\C^n))$ and in 
$\Lambda^k(S^m(\C^n))$.

\subsection{A basis for $\calq^{U_n}$}  

Since the action by $U_n$ and $\sym k$ on $\calq$ commute with each other, 
$\calq^{U_n}$ is a $\sym k$-module and 
we have
\begin{equation}\label{eq:qs3uncommute}
\bigl(\calq^{\sym k}\bigr)^{U_n}=\bigl(\calq^{U_n}\bigr)^{\sym k}
\quad\text{and}\quad
(\calq^{\sgn})^{U_n}=(\calq^{U_n})^{\sgn}.
\end{equation}
In this subsection, we shall describe a  basis for $\calq^{U_n}$. 
The standard monomial theory for $\Gln$ specifies a basis for 
the algebra $\calp^{U_n}$ of $U_n$-invariants in $\calp$,
and the elements of this basis are certain monomials on a set of generators called \emph{standard monomials}.
The basis elements which belong to the subalgebra $\calq^{U_n}$ forms a basis for $\calq^{U_n}$.

\medskip
Recall that $\calp=\pmnk$ is the algebra of all polynomial functions on the space $\Mnk=\Mnk(\C)$ of $n\times k$ complex matrices.
We shall write a typical element of $\Mnk$ as
\[
X=
\begin{pmatrix}
x_{11}&x_{12}&\cdots&x_{1k}\\
x_{21}&x_{22}&\cdots&x_{2k}\\
\vdots&\vdots&&\vdots\\
x_{n1}&x_{n2}&\cdots&x_{nk}
\end{pmatrix},
\]
so that $\pmnk$ can be identified as the polynomial algebra on the variables $x_{ij}$.
We note that for $g=(g_{ij})\in\GL_n$ and $h=(h_{ij})\in\GL_k$, we have
\begin{equation*}
(g,h).x_{ij}=\sum_{s=1}^n\sum_{t=1}^k g_{is}h_{tj}x_{st}
\end{equation*}
for $1\le i\le n$ and $1\le j\in k$.
In particular, we may identify a permutation $\tau\in\sym k$  with the permutation matrix $(\delta_{i\sigma(j)})$ in $\GL_k$,
which acts on $\calp$ by
\begin{equation*}
\tau.x_{ij}=x_{i\tau(j)}.
\end{equation*}

\medskip

For a Young diagram $D$ and $\bfm\in\Z^k_{\ge 0}$, let
$\SST(D,\bfm)$ denote the set of all semistandard tableaux of shape $D$ and content $\bfm$, and let
\[
K_{D,\bfm}=\text{the cardinality of $\SST(D,\bfm)$.}
\]
The nonnegative integer $K_{D,\bfm}$ is called a \emph{Kostka number} (\cite{F}).

\medskip

From now on, we shall assume that $n\ge k$. Let $D$ be a Young diagram with  $\ell(D)\le k$, and let 
$\bfm\in\Z^k_{\ge 0}$.
For each  $T\in\SST(D,\bfm)$, we  shall associate with $T$ a polynomial $\delta_T$ in $\calp$ as follows:
If $T$ consists of a single column with entries $I=(i_1,\dots,i_s)$ where $i_1<\cdots<i_s$, then we let
\[
\delta_T =
\begin{vmatrix}
x_{1i_1}&x_{1i_2}&\cdots&x_{1i_s}\\
x_{2i_1}&x_{2i_2}&\cdots&x_{2i_s}\\
\vdots&\vdots&&\vdots\\
x_{si_1}&x_{si_2}&\cdots&x_{si_s}
\end{vmatrix}.
\]
In general, if $T_1,T_2,\dots,T_p$ are the columns of $T$ from left to right, then we define
\[\delta_T=\delta_{T_1}\delta_{T_2}\cdots \delta_{T_p}.\]
We also let $m_T$ be the monomial in $\calp$ given by 
\[m_T=\prod_{i,j}x_{ij}^{e_{ij}}\]
where
\[e_{ij}=\text{the number of boxes in the $i$-th row of $T$ labeled with the number $j$}.\]

\medskip

Next, we let the set of monomials in $\calp$ be given the \emph{graded lexicographic order} (\cite{CLO}) such that 
\[
x_{ab}>x_{cd}
\quad \text{if and only if}\quad
\begin{cases}
\;\text{either  (i)  $b<d$ }\\[.25em]
\;\text{or (ii) $b=d$ and $a<c$.}
\end{cases}
\]
Thus under this monomial ordering, we have
\[x_{11}>x_{21}>\cdots>x_{n1}>x_{12}>x_{22}>\cdots>x_{nk}.\]
If $f$ is a nonzero element of $\calp$, then we shall denotes the leading monomial of $f$
with respect to the above monomial ordering by $\LM(f)$.
Notice that $\LM(fg)=\LM(f)\LM(g)$ for any nonzero elements $f,g$ of $\calp$.

\medskip

The following results are well known (\cite{HKL,HL,Le}). 

\begin{prop} \label{qmdbasis}
Let $\bfm=(m_1,\dots,m_k)\in\Z^k_{\ge 0}$.
\begin{enumerate}[(i)]
\item
As a representation of $\Gln$,
\[
\calp_{\bfm}
\cong S^{m_1}(\C^n)\otimes\cdots\otimes S^{m_k}(\C^n)
=\bigoplus_{\ell(D)\le k}K_{D,\bfm}\,\rho^D_n.
\]

\item
If $T\in\SST(D,\bfm)$, then 
\[
\LM(\delta_T)=m_T.
\]
Moreover, the set 
\begin{equation}\label{eq:bmd}
\bbf_{\bfm,D}=\set{\delta_T}{T\in \SST(D,\bfm)}
\end{equation}
is a basis for $\calp^{U_n}_{\bfm, D}$. 
\end{enumerate}
\end{prop}

Recall that for each $m\in\Z_{\ge 0}$, $\calq_m=\calp_{(m^k)}$.
If $D$ is a Young diagram with $\ell(D)\le k$, then we let $\calq_{m,D}^{U_n}=\calp^{U_n}_{(m^k),D}$.
Then $\calq^{U_n}$ has a decomposition 
\begin{equation}\label{eq:qun}
\calq^{U_n}=\bigoplus_{\substack{m\ge0 \\ \ell(D)\le k}}\calq_{m,D}^{U_n}.
\end{equation}

\begin{cor}\label{cor:bmkd}
For each $m\in\Z_{\ge 0}$ and for each Young diagram $D$ with $\ell(D)\le k$,
the set $\bbf_{(m^k),D}$ is a basis for $\calq_{m,D}^{U_n}$. Consequently, the union
\[\bbf_{\calq^{U_n}}=\bigcup_{\stackrel{m\in\Z_{\ge 0}}{\ell(D)\le k }}\bbf_{(m^k),D}\]
is a basis for $\calq^{U_n}$.
\end{cor}

\section{The case $m=1$ --- Schur-Weyl duality}\label{sec:schurweyl}

For a Young diagram $D$ with $k$ boxes,
we denote by $\symirrep{k}{D}$ the irreducible $\sym k$-module corresponding to $D$.
Then by Schur-Weyl duality (\cite{H2,GW}),
\begin{equation*}
\cQ_1\cong(\C^n)^{\otimes k}\cong\bigoplus_{ |D|=k  }\rho_n^D\otimes\symirrep{k}{D}
\end{equation*}
as a representation for $\Gln\times\sym k$. Consequently,
\begin{equation*}
\cQ_1^{U_n}=\bigoplus_{ |D|=k  }\cQ^{U_n}_{1,D} \cong\bigoplus_{ |D|=k  } (\rho_n^D)^{U_n}\otimes\symirrep{k}{D}.
\end{equation*}
In particular, for each Young diagram $D$ with $k$ boxes,
we can realize the irreducible $\sym k$ module $\symirrep{k}{D}$   as the space $\cQ^{U_n}_{1,D}$ and 
the set $\set{\delta_T}{T\in\SST(D,(1^k))}$ forms a basis for $\symirrep{k}{D}$.

\begin{remark}
Consider the $\C$-algebra homomorphism $\calp\to\C[z_1,\dots,z_k]$ defined by $x_{ij}\mapsto z_j^{i-1}$.
By this homomorphism, for each standard tableau $T\in\SST(D,(1^k))$ on a $k$-box Young diagram $D$,
$\delta_T$ is sent to the \emph{Specht polynomial}
\begin{equation*}
F_T(z_1,\dots,z_k):=\prod_{j=1}^{\len\mu}\prod_{1\le p<q\le \mu_j}(z_{T(q,j)}-z_{T(p,j)}),
\end{equation*}
where $\mu$ is a partition associated to the conjugate diagram of $D$.
It is well known that
the Specht polynomials $F_T(z_1,\dots,z_k)$ $(T\in\SST(D,(1^k)))$
form a basis of an irreducible $\sym k$-submodule in $\C[z_1,\dots,z_k]$
isomorphic to $\symirrep{k}{D}$ (see, e.g. \cite{F}).
In this sense, our $\delta_T$ is regarded as a refinement of the Specht polynomial $F_T$.
\end{remark}

\section{The case $k=2$}

In this section, we determine all the $\Gln$-highest weight vectors which occur in 
$S^2(S^m(\C^n))$ and in $\Lambda^2(S^m(\C^n)$. 

\begin{prop}
Assume that $k=2$ and let 
\begin{equation*}
\alpha=x_{11}x_{12},\quad
\gamma=\delta_{12}.
\end{equation*}
\begin{enumerate}
\item[(i)] 
The algebra $\calq^{U_n}$ is generated by $\{\alpha,\gamma\}$
 and the set   $\set{\alpha^{m-a}\gamma^a}{0\le a\le m}$ forms a basis for $\calq^{U_n}$.
 \item[(ii)] For each positive integer $m$,  we have
 \begin{equation*}
 \calq_m^{\sym2}\cong S^2(S^m(\C^n))
\cong\bigoplus_{\substack{0\le a\le m\\\text{$a$:\,even}}}\rho_n^{(2m-a,a)},
 \end{equation*}
 and for each even integer $a$ such that $0\le a\le m$, $\alpha^{m-a}\gamma^a$ is a $\Gln$-highest weight vector in 
 $\rho_n^{(2m-a,a)}$ (which is unique up to a scalar multiple).
   \item[(iii)] For each positive integer $m$,  we have
 \begin{equation*}
 \calq_m^{\sgn}\cong \Lambda^2(S^m(\C^n))
\cong\bigoplus_{\substack{1\le a\le m\\\text{$a$:\,odd}}}\rho_n^{(2m-a,a)}. 
 \end{equation*}
 and for each odd integer $a$ such that $1\le a\le m$, $\alpha^{m-a}\gamma^a$ is a $\Gln$-highest weight vector in 
 $\rho_n^{(2m-a,a)}$ (which is unique up to a scalar multiple).
 \item[(iv)] We have
 \begin{equation*}
\calq^{\sym2}=\C[\alpha,\gamma^2],\quad
\calq^{\sgn}=\gamma\calq^{\sym2}=\gamma\C[\alpha,\gamma^2].
\end{equation*}
\end{enumerate}
\end{prop}

\begin{proof}
Fix a positive integer $m$ and a Young diagram $D$ with $\ell(D)\le 2$ and $|D|=2m$.
We shall show that for each $T\in \SST(D,(m^2))$,
$\delta_T$ is of the form
\begin{equation}
\delta_T=\alpha^a\gamma^c
\end{equation}
where $a,c\in\Z_{\ge 0}$.

\medskip
If $D$ has only one row and $T\in \SST(D,(m^2))$, then it is clear that $\delta_T=\alpha^m$.

\medskip
If $D$ has two rows, then $D=(2m-a,a)$ for some $1\le a\le m$. 
In this case, if $T\in\SST(D,(m^2))$,
then $T$ must be the tableau
\begin{equation*}
T=\vcentralize{\includegraphics{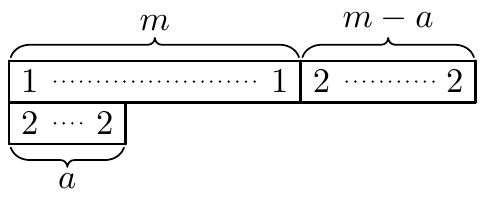}},
\end{equation*}
and so
\begin{equation*}
\delta_T=\delta_{12}^a(x_{11}x_{12})^{m-a}=\alpha^{m-a}\gamma^a.
\end{equation*}
Thus,
\[\bbf_{(m^2),D}=\set{\alpha^{m-a}\gamma^a}{0\le a\le m}\]
is a basis for $(\calq_m)^{U_n}$
and 
\begin{equation*}
\calq_m\cong S^m(\C^n)\otimes S^m(\C^n)\cong\bigoplus_{0\le a\le m}\rho_n^{(2m-a,a)}.
\end{equation*}
Part (i) follows from this and Corollary \ref{cor:bmkd}.

\medskip
Next, since
 \begin{equation*}
(1\,2). \alpha^{m-a}\gamma^a=(-1)^a \alpha^{m-a}\gamma^a,
\end{equation*}
we have
\begin{equation*}
\alpha^{m-a}\gamma^a\in\begin{cases}
(\calq_m^{\sym2})^{U_n} & \text{if $a$ is even}, \\
(\calq_m^{\sgn})^{U_n} & \text{if $a$ is odd}.
\end{cases}
\end{equation*}
It is clear that (ii), (iii) and (iv) follow from this. 
\end{proof}

\section{The case $k=3$}

In this section, we will implement the procedure outlined in \S \ref{sec:uninv}
for the case $k=3$. Specifically, we shall determine a basis for the algebra
$(\calq^{\sym k}_m)^{U_n}$ and a basis for the space $(\calq^{\sgn}_m)^{U_n}$.
Using these bases, we obtain all the $\Gln$-highest weight vectors which occur in 
$S^3(S^m(\C^n))$ and in $\Lambda^3(S^m(\C^n))$.

\subsection{Algebra generators for $\calq^{U_n}$}

In this subsection, we shall describe a finite set of generators for the algebra $\calq^{U_n}$. 

\begin{prop}
The algebra $\calq^{U_n}$ is generated by
\[
\alpha_1=x_{11}x_{12}x_{13},\quad
\beta_2=x_{12}\delta_{13},\quad
\beta_3=x_{13}\delta_{12},\quad
\gamma_1=
\begin{vmatrix}
x_{11}&x_{12}&x_{13}\\
x_{21}&x_{22}&x_{23}\\
x_{31}&x_{32}&x_{33}
\end{vmatrix},\quad
\gamma_2=\delta_{12}\delta_{13}\delta_{23}.
\] 
\end{prop}

\begin{proof}
Fix a positive integer $m$ and a Young diagram $D$ with  $\ell(D)\le 3$ and $|D|=3m$.
We shall show that for each $T\in \SST(D,(m^3))$, 
$\delta_T$ can be expressed of the form
\begin{equation}\label{eq:ftform}
\delta_T=\alpha_1^a\beta_2^{b_2}\beta_3^{b_3}\gamma_1^{c_1}\gamma_2^{c_2}
\end{equation}
where $a,b_1,b_2,c_1,c_2\in\Z_{\ge 0}$.
We consider several cases according to the number of rows in $D$.

\medskip
Case I: $\len D=1$.
We must have $D=(3m)$, and $\SST((3m),(m^3))$ consists of exactly one tableau $T$ which is given by 
\begin{equation*}
T=\lower.5em\hbox{\includegraphics{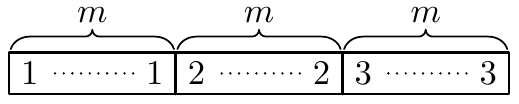}}.
\end{equation*}
Then $\delta_T=\alpha^m_1$.
Note that $\alpha_1$ is an invariant for $\sym3$.

\medskip
Case II: $\len D=2$.
Let us take a tableau $T\in\SST(D,(m^3))$,
and suppose that there are $a$ $2$'s and $b$ $3$'s
in the second row of $T$.
If $a+b\le m$, then $T$ is of the form
\begin{equation*}
T=\vcentralize{\includegraphics{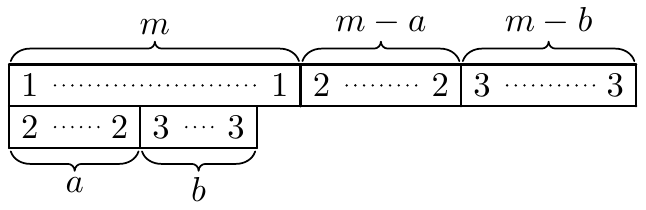}},
\end{equation*}
and we have
\begin{align*}
\delta_T
&=\delta_{12}^a \delta_{13}^b x_{11}^{m-a-b} x_{12}^{m-a} x_{13}^{m-b} \\
&=(x_{13}\delta_{12})^a (x_{12}\delta_{13})^b (x_{11}x_{12}x_{13})^{m-a-b} \\
&=\beta_2^a \beta_3^b \alpha_1^{m-a-b}.
\end{align*}
If $a+b>m$, then $T$ is of the form
\begin{equation*}
T=\vcentralize{\includegraphics{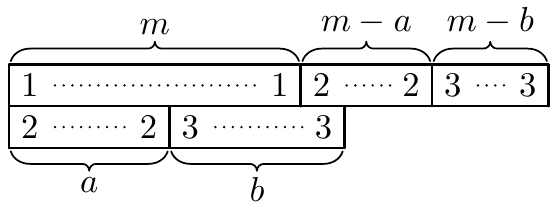}},
\end{equation*}
and we have
\begin{align*}
\delta_T
&=\delta_{12}^a \delta_{13}^{m-a} \delta_{23}^{a+b-m} x_{12}^{(2m-a)-(a+b)} x_{13}^{m-b} \\
&=(\delta_{12}\delta_{13}\delta_{23})^{a+b-m}(x_{13}\delta_{12})^{m-b}(x_{12}\delta_{13})^{2m-2a-b} \\
&=\gamma_2^{a+b-m}\beta_2^{m-b}\beta_3^{2m-2a-b}.
\end{align*}
We see in both cases, $\delta_T$ is expressed as a monomial in $\alpha_1,\beta_2,\beta_3$ and $\gamma_2$.

\medskip
Case III: $\len D=3$.
Let us take a tableau $T\in\SST(D,(m^3))$.
Assume that there are $e$ boxes in the bottom row of $D$,
and let $\widehat{D}$ be a Young diagram obtained from $D$ by deleting first $e$ columns.
Then $\len{\widehat{D}}\le2$ and 
$T$ is of the form
\begin{equation*}
T=\lower1.5em\hbox{\includegraphics{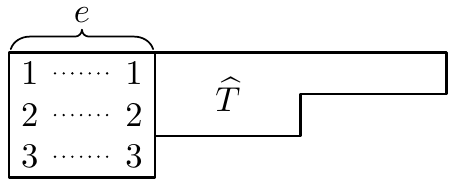}},
\end{equation*}
where $\widehat{T}\in\SST(\widehat{D},(m-e)^3)$.
Thus we have
\begin{equation*}
\delta_T=\gamma_1^e\delta_{\widehat{T}}.
\end{equation*}
By the preceding discussion,
$\delta_{\widehat{T}}$ is in form \eqref{eq:ftform}.
Hence, $\delta_T$ is also in this form.
\end{proof}

\begin{remark}
It follows from the discussion above that
\begin{equation*}
K_{D,(m^3)}=\min\{\lambda_1-\lambda_2,\lambda_2-\lambda_3\}+1
\end{equation*}
for each Young diagram $D=(\lambda_1,\lambda_2,\lambda_3)$ with $|D|=3m$.
\end{remark}

\subsection{Two subalgebras of $\calq^{U_n}$} 

Our goal is to determine a basis for $(\calq^{U_n})^{\sym3}$ and a basis for $(\calq^{U_n})^{\sgn}$.
In the previous subsection, we proved that the algebra $\calq^{U_n}$ is generated by $\alpha_1,\gamma_1,\gamma_2, \beta_2,\beta_3$. 
Let $\calr_1$ and $\calr_2$ be the subalgebras of $\calq^{U_n}$ defined by 
\[
\calr_1=\C[\alpha_1,\gamma_1,\gamma_2],\quad \calr_2=\C[\beta_2,\beta_3].
\]
We note that $\calr_1$ and $\calr_2$ are also $\sym3$-modules.
In fact, $\alpha_1$ is a $\sym3$-invariant,
and $\gamma_1$ and $\gamma_2$ each generates a one-dimensional representation of $\sym3$
which is isomorphic to the sign representation.
Therefore, it is easy to describe the $\sym3$-module structure of $\calr_1$ and we will do this later.
On the other hand, $\beta_2$ and $\beta_3$ span the irreducible $\sym3$-module $\symirrep{3}{(2,1)}$
labeled by the Young diagram $(2,1)$ (see \S \ref{sec:schurweyl}),
so that $\calr_2$ is the symmetric algebra on $\symirrep{3}{(2,1)}$.
To determine the structure of $(\calq^{U_n})^{\sym3}$ and $(\calq^{U_n})^{\sgn}$,
the main work is to determine the structure of $(\calr_2)^{\sym3}$ and $(\calr_2)^{\sgn}$.

\medskip
We now consider a slightly general situation.
Let $k\in\Z$ be such that $k\ge3$.
The polynomial algebra $\C[t_1,\dots,t_k]$ is a $\sym k$-module
by letting $\sigma\cdot t_i=t_{\sigma(i)}$ for $\sigma\in\sym k$ and $i=1,\dots,k$.
It is well known that the algebra $\C[t_1,\dots,t_k]^{\sym k}$ of $\sym k$-invariants,
or \emph{symmetric polyonomials} in $k$ variables,
is a polynomial algebra in the \emph{elementary symmetric polynomials}
\begin{equation*}
\sigma_s=\sum_{1\le j_1<\dots<j_s\le k}t_{j_1}\dots t_{j_s}\quad(s=1,\dots,k),
\end{equation*}
which are algebraically independent.
Namely, $\C[t_1,\dots,t_k]^{\sym k}=\C[\sigma_1,\dots,\sigma_k]$.
We also notice that the $\sym k$-submodule $\C[t_1,\dots,t_k]^{\sgn}$
consisting of \emph{alternating polynomials} is given by
\begin{equation*}
\C[t_1,\dots,t_k]^{\sgn}=\Delta\cdot\C[\sigma_1,\dots,\sigma_k],
\end{equation*}
where
\begin{equation*}
\Delta=\det(t_j^{i-1})_{1\le i,j\le k}=\prod_{1\le i<j\le k}(t_i-t_j)
\end{equation*}
is the simplest alternating polynomial.

\medskip
For each $i=2,3,\dots,k$, put
\begin{equation*}
\beta_i=
x_{12}\dots\om{x_{1i}}\dots x_{1k}
\begin{vmatrix}
x_{11} & x_{1i} \\
x_{21} & x_{2i}
\end{vmatrix},
\end{equation*}
where $\om{x_{1i}}$ indicates the omission of $x_{1i}$ in the product.

\begin{lemma}
For $\sigma\in\sym k$ and $i=2,\dots,k$,
\begin{equation*}
\sigma\cdot\beta_i=\begin{cases}
\beta_{\sigma(i)} & \sigma(1)=1, \\
-\beta_{\sigma(1)} & \sigma(i)=1, \\
\beta_{\sigma(i)}-\beta_{\sigma(1)} & \sigma(j)=1,\;j\ne 1,i.
\end{cases}
\end{equation*}
\end{lemma}

\begin{proof}
Let $\sigma\in\sym k$.
If $\sigma(1)=1$, then
\begin{equation*}
\sigma\cdot\beta_i
=x_{1\sigma(2)}\dots\om{x_{1\sigma(i)}}\dots x_{1\sigma(k)}
\begin{vmatrix}
x_{11} & x_{1\sigma(i)} \\
x_{21} & x_{2\sigma(i)}
\end{vmatrix}
=\beta_{\sigma(i)}.
\end{equation*}
Next, assume that $\sigma(1)\ne1$.
If $\sigma(i)=1$, then
\begin{equation*}
\sigma\cdot\beta_i
=x_{1\sigma(2)}\dots\om{x_{1\sigma(i)}}\dots x_{1\sigma(k)}
\begin{vmatrix}
x_{1\sigma(1)} & x_{11} \\
x_{2\sigma(1)} & x_{21}
\end{vmatrix}
=-\beta_{\sigma(1)}.
\end{equation*}
If $\sigma(j)=1$ for $j\ne i$, then
\begin{align*}
\sigma\cdot\beta_i
&=x_{1\sigma(2)}\dots\om{x_{1\sigma(i)}}\dots{x_{1\sigma(j)}}\dots x_{1\sigma(k)}
\begin{vmatrix}
x_{1\sigma(1)} & x_{1\sigma(i)} \\
x_{2\sigma(1)} & x_{2\sigma(i)}
\end{vmatrix}
\\
&=x_{1\sigma(1)}x_{1\sigma(2)}\dots\om{x_{1\sigma(i)}}\dots\om{x_{1\sigma(j)}}\dots x_{1\sigma(k)}
\cdot x_{11}x_{2\sigma(i)} \\
&\qquad{}-x_{1\sigma(2)}\dots{x_{1\sigma(i)}}\dots\om{x_{1\sigma(j)}}\dots x_{1\sigma(k)}
\cdot x_{11}x_{2\sigma(1)} \\
&=\beta_{\sigma(i)}-\beta_{\sigma(1)}.
\qedhere
\end{align*}
\end{proof}

Put
\begin{equation*}
T_1=\beta_2+\beta_3+\dots+\beta_k.
\end{equation*}
By the lemma above, for each $\sigma\in\sym k$, we have
\begin{equation*}
\sigma\cdot T_1=\begin{cases}
T_1 & \sigma(1)=1, \\
T_1-k\beta_{\sigma(1)} & \sigma(1)\ne1.
\end{cases}
\end{equation*}
Hence, if we put
\begin{equation*}
T_i=T_1-k\beta_i\qquad(i=2,3,\dots,k),
\end{equation*}
then we have
\begin{equation*}
\sigma\cdot T_i=T_{\sigma(i)}
\end{equation*}
for any $\sigma\in\sym k$ and $i=1,2,\dots,k$.
This implies that the $\C$-algebra homomorphism defined by
\begin{equation*}
\Phi\colon\C[t_1,\dots,t_k]\ni t_i\mapsto T_i\in\C[\beta_2,\dots,\beta_k]
\end{equation*}
is a surjective $\sym k$-map such that $\Phi(\sigma_1)=0$.
Consequently, we have
\begin{multline*}
\C[\beta_2,\dots,\beta_k]^{\sym k}
=\Phi(\C[t_1,\dots,t_k])^{\sym k}
=\Phi(\C[t_1,\dots,t_k]^{\sym k}) \\
=\Phi(\C[\sigma_1,\dots,\sigma_k])
=\C[\Phi(\sigma_2),\dots,\Phi(\sigma_k)]
\end{multline*}
and
\begin{equation*}
\C[\beta_2,\dots,\beta_k]^{\sgn}
=\Phi(\C[t_1,\dots,t_k]^{\sgn})
=\Phi(\Delta)\cdot\C[\Phi(\sigma_2),\dots,\Phi(\sigma_k)].
\end{equation*}

\medskip 
We now return to the case $k=3$.
The images of $\sigma_2, \sigma_3$ and $\Delta$ under $\Phi$ are given by
\begin{align*}
\Phi(\sigma_2)&=-3(\beta_2^2-\beta_2\beta_3+\beta_3^2)=-3\alpha_2,\\
\Phi(\sigma_3)&=-(2\beta^3_2-3\beta^2_2\beta_3-3\beta_2\beta^2_3+2\beta^3_3)=-\alpha_3,\\
\Phi(\Delta)&=27\alpha_1\gamma_1.
\end{align*}
The discriminant $\Delta^2$ of the cubic polynomial
\begin{equation*}
f(z)=(z-t_1)(z-t_2)(z-t_3)=z^3-\sigma_1z^2+\sigma_2z-\sigma_3
\end{equation*}
is explicitly given by
\begin{equation*}
\Delta^2=-4\sigma_2^3-27\sigma_3^2+\sigma_1^2\sigma_2^2+18\sigma_1\sigma_2\sigma_3-4\sigma_1^3\sigma_3.
\end{equation*}
Thus we have the relation
\begin{equation*}
\Phi(\Delta)^2=-4\Phi(\sigma_2)^3-27\Phi(\sigma_3)^2=27(4\alpha_2^3-\alpha_3^2).
\end{equation*}
Summarizing the discussion above, we have the
\begin{lemma}\label{q2s3inv}
The subalgebra $(\calr_2)^{\sym3}$ and the subspace $(\calr_2)^{\sgn}$ are given by
\begin{equation*}
(\calr_2)^{\sym3}=\C[\beta_2,\beta_3]^{\sym3}=\C[\alpha_2,\alpha_3],\qquad
(\calr_2)^{\sgn}=\alpha_1\gamma_1\cdot\C[\alpha_2,\alpha_3].
\end{equation*}
Moreover, we have
\begin{equation}\label{eq:alpha23}
27\alpha_1^2\gamma_1^2=4\alpha_2^3-\alpha_3^2.
\end{equation}
\end{lemma}

\subsection{$\sym3$ invariants in $\calq^{U_n}$}

We now consider the tensor product 
$\calr_1\otimes \calr_2$ of the algebras $\calr_1$ and $\calr_2$, and 
let $m\colon\calr_1\otimes\calr_2\to\calq^{U_n}$ be the linear map such that for $(f_1,f_2)\in\calr_1\times\calr_2$,
\[
m(f\otimes g)=fg
\]
where $fg$ is the product of $f$ and $g$ in $\calq^{U_n}$.
The map $m$ is clearly surjective.
Moreover, since $\sym3$ acts on $\calq^{U_n}$ by algebra automorphisms, $m$ is a $\sym3$-module map.
Hence we obtain the following lemma:

\begin{lemma}
Let $(\calr_1\otimes\calr_2)^{\sym3}$ be the space of $\sym3$-invariants in $\calr_1\otimes\calr_2$.
Then
\[
m\bigl((\calr_1\otimes\calr_2)^{\sym3}\bigr)=\bigl(\calq^{U_n}\bigr)^{\sym3}.
\eqed
\]
\end{lemma}

\begin{lemma}\label{calq1}
Let
\[
\bbf_{\calr_1}=\set[big]{\alpha^a_1\gamma^{b_1}_1\gamma^{b_2}_2}{a,b_1,b_2\in\Z_{\ge 0}}.
\]
\begin{enumerate}
\item[(i)] 
The set $\bbf_{\calr_1}$ is a basis for $\calr_1$.
\item[(ii)] The algebra $\calr_1$ has a decomposition
\[
\calr_1=\bigoplus_{a,b_1,b_2\in\Z_{\ge 0}} (\calr_1)_{(a,b_1,b_2)},
\]
where $(\calr_1)_{(a,b_1,b_2)}$ is the subspace of $\calr_1$ spanned by 
$\alpha^a_1\gamma^{b_1}_1\gamma^{b_2}_2$.
Moreover, the subspace
$(\calr_1)_{(a,b_1,b_2)}$ is an irreducible $\sym3$-submodule and
\[
(\calr_1)_{(a,b_1,b_2)}\cong
\begin{cases}
\triv & \text{if $b_1+b_2$ is even}, \\
\sgn & \text{if $b_1+b_2$ is odd.}
\end{cases}
\]
\end{enumerate}
\end{lemma}

\begin{proof}
Since $\alpha_1,\gamma_1,\gamma_2$ generates the algebra $\calr_1$, $\bbf_{\calr_1}$ spans $\calr_1$. 
We compute
\begin{align*}
\LM(\alpha^a\gamma^{b_1}_1\gamma^{b_2}_2)
&=(x_{11}x_{12}x_{13})^a(x_{11}x_{22}x_{33})^{b_1}(x_{11}x_{22}x_{11}x_{23}x_{12}x_{23})^{b_2}\\
&=x_{11}^{a+b_1+2b_2}x_{12}^{a+b_2}x_{13}^ax_{22}^{b_1+b_2}
x_{33}^{b_1}x_{23}^{2b_2}
\end{align*}
and note that the exponents of $x_{13}$, $x_{33}$ and $x_{23}$ uniquely determine $a,b_1$ and $b_2$.
Hence, the elements of $\bbf_{\calr_1}$ have distinct leading monomials,
and it follows from this that $\bbf_{\calr_1}$ is linearly independent.
This gives (i). Part (ii) clearly follows from (i).
\end{proof}

Next, suppose that
\[
\calr_2=\bigoplus_i V_i
\]
is a decomposition of $\calr_2$ into irreducible $\sym3$-submodules. Then
\[
\calr_1\otimes \calr_2=\bigoplus_{a,b_1,b_2,i} (\calr_1)_{(a,b_1,b_2)}\otimes V_i
\]
and so we have
\begin{align*}
(\calr_1\otimes \calr_2)^{\sym3}
&=\bigoplus_{a,b_1,b_2,i}\bigl( (\calr_1)_{(a,b_1,b_2)}\otimes V_i\bigr)^{\sym3}, \\
(\calr_1\otimes \calr_2)^{\sgn}
&=\bigoplus_{a,b_1,b_2,i}\bigl( (\calr_1)_{(a,b_1,b_2)}\otimes V_i\bigr)^{\sgn}.
\end{align*}
Since
\begin{gather*}
\triv\otimes V_i\cong V_i, \\
\sgn\otimes \sgn\cong \triv, \quad
\sgn\otimes \triv\cong \sgn, \quad
\sgn\otimes\symirrep{3}{(2,1)}\cong\symirrep{3}{(2,1)},
\end{gather*}
we see that 
\begin{equation}\label{eq:rs3inv}
\bigl( (\calr_1)_{(a,b_1,b_2)}\otimes V_i\bigr)^{\sym3}\ne 0
\iff
\begin{cases}
\text{Case 1:} & (\calr_1)_{(a,b_1,b_2)}\cong  V_i\cong\triv,\ \text{or} \\[.25em]
\text{Case 2:} & (\calr_1)_{(a,b_1,b_2)}\cong  V_i\cong\sgn
\end{cases}
\end{equation}
and 
\begin{equation}\label{eq:rs3sgn}
\bigl( (\calr_1)_{(a,b_1,b_2)}\otimes V_i\bigr)^{\sgn}\ne 0
\iff
\begin{cases}
\text{Case 3:} & (\calr_1)_{(a,b_1,b_2)}\cong\sgn \text{\ and\ }  V_i\cong\triv,\ \text{or}\\[.25em]
\text{Case 4:} & (\calr_1)_{(a,b_1,b_2)}\cong\triv\text{\ and }\cong  V_i\cong\sgn.
\end{cases}
\end{equation}
It follows that  we may replace the algebra $\calr_2$ in the tensor product $\calr_1\otimes\calr_2$
by the subspace  $\calr_2^{\sym3}\oplus\calr_2^{\sgn}$, that is, 
\[
(\calr_1\otimes \calr_2)^{\sym3}=(\calr_1\otimes (\calr_2^{\sym3}\oplus\calr_2^{\sgn}))^{\sym3},
\]
and 
\[
(\calr_1\otimes \calr_2)^{\sgn}=(\calr_1\otimes (\calr_2^{\sym3}\oplus\calr_2^{\sgn}))^{\sgn}.
\]
Now $V_i\cong\triv$ if and only if $V_i=\Span(f)$ for some $f\in (\calr_2)^{\sym3}$,
and by Lemma \ref{q2s3inv}, $f$ is a polynomial in $\alpha_2$ and $\alpha_3$.
In fact, we may assume that $f=\alpha^c_2\alpha^d_3$ for some $c,d\in\Z_{\ge 0}$. 
Similarly, $V_i\cong\sgn$ if and only if $V_i=\Span(f)$
where $f=\alpha_1\gamma_2 h$ for some $h\in (\calr_2)^{\sym3}$,
and so we may assume that $f=\alpha_1\alpha^c_2\alpha^d_3\gamma_2$.

\medskip
We can now describe the vector in the image under the map $m$
of $( (\calr_1)_{(a,b_1,b_2)}\otimes V_i)^{\sym3}$ in Case 1 and Case 2,
and of $( (\calr_1)_{(a,b_1,b_2)}\otimes V_i)^{\sgn}$ in Case 3 and Case 4.

\medskip
\begin{enumerate}[{Case} 1:]
\item
Since $b_1+b_2$ is even, $b_1=2j_1+\ve$ and $b_2=2j_2+\ve$ for some $j_1,j_2\in\Z_{\ge 0}$ and $\ve\in\{0,1\}$.
In this case,
\[
m(\alpha^a_1\gamma^{b_1}_1\gamma^{b_2}_2\otimes \alpha^c_2\alpha^d_3)=
\alpha^a_1 \alpha^c_2\alpha^d_3 (\gamma^2_1)^{j_1}_1(\gamma^2_2)^{j_2}(\gamma_1\gamma_2)^\ve.
\]

\item
Since $b_1+b_2$ is odd, we either have  $b_1=2j_1+1$ and $b_2=2j_2$ or 
$b_1=2j_1$ and $b_2=2j_2+1$ for some $j_1,j_2\in\Z_{\ge 0}$.
So we either have  
\[
m(\alpha^a_1\gamma^{b_1}_1\gamma^{b_2}_2\otimes \alpha_1\alpha^c_2\alpha^d_3\gamma_2)=
\alpha^{a+1}_1 \alpha^c_2\alpha^d_3 (\gamma^2_1)^{j_1}(\gamma^2_2)^{j_2}(\gamma_1\gamma_2)
\]
or 
\[
m(\alpha^a_1\gamma^{b_1}_1\gamma^{b_2}_2\otimes \alpha_1\alpha^c_2\alpha^d_3\gamma_2)=
\alpha^{a+1}_1 \alpha^c_2\alpha^d_3 (\gamma^2_1)^{j_1}(\gamma^2_2)^{j_2+1}.
\]

\item
Since $b_1+b_2$ is odd, we either have  $b_1=2j_1+1$ and $b_2=2j_2$ or 
$b_1=2j_1$ and $b_2=2j_2+1$ for some $j_1,j_2\in\Z_{\ge 0}$. So we either have
\[
m(\alpha^a_1\gamma^{b_1}_1\gamma^{b_2}_2\otimes  \alpha^c_2\alpha^d_3 )=
\alpha^{a}_1 \alpha^c_2\alpha^d_3 (\gamma^2_1)^{j_1}(\gamma^2_2)^{j_2}\gamma_1
\]
or 
\[
m(\alpha^a_1\gamma^{b_1}_1\gamma^{b_2}_2\otimes \alpha^c_2\alpha^d_3 )=
\alpha^{a}_1 \alpha^c_2\alpha^d_3 (\gamma^2_1)^{j_1}(\gamma^2_2)^{j_2}\gamma_2.
\]

\item
Since $b_1+b_2$ is even, $b_1=2j_1+\ve$ and $b_2=2j_2+\ve$ for some $j_1,j_2\in\Z_{\ge 0}$ and $\ve\in\{0,1\}$.
In this case, 
\[
m(\alpha^a_1\gamma^{b_1}_1\gamma^{b_2}_2\otimes \alpha_1\alpha^c_2\alpha^d_3\gamma_2)=
\alpha^{a+1}_1 \alpha^c_2\alpha^d_3 (\gamma^2_1)^{j_1}_1(\gamma^2_2)^{j_2}(\gamma_1\gamma_2)^\ve\gamma_2.
\]
\end{enumerate}
Thus we have proved:

\begin{prop}\label{prop:alggen}
\begin{enumerate}[(i)]
\item The algebra  $(\calq^{U_n})^{\sym3}$ is spanned by the set 
\[
\Sinv:=\set[big]{\alpha^a_1\alpha^b_2\alpha^c_3(\gamma^2_1)^d(\gamma^2_2)^e(\gamma_1\gamma_2)^f}
{a,b,c,d,e,f\in\Z_{\ge 0}}.
\]
Hence, 
$\{\alpha_1, \alpha_2, \alpha_3, \gamma_1^2, \gamma_2^2, \gamma_1\gamma_2\}$
is a set of algebra generators for $(\calq^{U_n})^{\sym3}$.
\item The space $(\calq^{U_n})^{\sgn}$ is spanned by the set 
\begin{align*}
\Ssgn&=\set{\gamma_1h}{h\in\Sinv}\cup \set{\gamma_2h}{h\in\Sinv}\\
&=\set[big]{\alpha^a_1\alpha^b_2\alpha^c_3\gamma^{2d+1}_1\gamma^{2e}_2}
{a,b,c,d,e\in\Z_{\ge 0}} \\
&\qquad{}\cup
\set[big]{\alpha^a_1\alpha^b_2\alpha^c_3\gamma_1^{2d}\gamma_2^{2e+1}}{a,b,c,d,e\in\Z_{\ge 0}}.
\eqed
\end{align*}
\end{enumerate}
\end{prop}

We are now ready to prove the Main Theorem as stated in the Introduction.

\begin{proof}[Proof of Main Theorem]
We will only prove (i) as the proof for (ii) is similar.
By Proposition \ref{prop:alggen}, $(\calq^{U_n})^{\sym3}$ is spanned by all elements of the form 
\[
\alpha^a_1\alpha^b_2\alpha^c_3\gamma_1^{2d+f}\gamma_2^{2e+f}
=\alpha^a_1\alpha^b_2\alpha^c_3(\gamma^2_1)^d(\gamma^2_2)^e(\gamma_1\gamma_2)^f
\]
with $a,b,c,d,e,f\in\Z_{\ge 0}$.
We need to explain why it is enough to have $c,f\in\{0,1\}$.

\medskip
By equation \eqref{eq:alpha23},
$4\alpha^3_2-\alpha^2_3=27\alpha^2_1\gamma^2_2$. So if $c=2j+\ve$ with $\ve=0$ or $1$, then 
\[\alpha^c_3=(\alpha^2_3)^j\alpha^\ve_3
=( 4\alpha^3_2-27\alpha^2_1\gamma^2_2)^j\alpha^\ve_3.
\]
Hence, it is enough to include those elements with
$c=0,1$. 

\medskip
We also have 
$(\gamma_1\gamma_2)^2=\gamma^2_1\gamma^2_2$.
If $f=2j+\ve$ with $\ve=0$ or $1$, then
\[
(\gamma_1\gamma_2)^f
=(\gamma^2_1)^j(\gamma^2_2)^j(\gamma_1\gamma_2)^\ve.
\]
Thus we may assume $f=0$ or $1$.
This shows that $\Binv$ spans $(\calq^{U_n})^{\sym3}$.

\medskip
Next we show that $\Binv$ is linearly independent.
It suffices to show that the elements of $\Binv$ have distinct leading monomials.
First, we list the leading monomials of
$\alpha_1$, $\alpha_2$, $\alpha_3$, $\gamma_1$ and $\gamma_2$: 
\begin{equation*}
\begin{array}{|c|c|}\hline
f &\LM(f)\\ \hline
\alpha_1&x_{11}x_{12}x_{13}\\ \hline
\alpha_2&x^2_{11}x^2_{12}x^2_{23}\\ \hline
\alpha_3&x^3_{11}x^3_{12}x^3_{23}\\ \hline
\gamma_1&x_{11}x_{22}x_{33}\\ \hline
\gamma_2&x^2_{11}x_{12}x_{22}x^2_{23}\\ \hline
\end{array}
\end{equation*}

We now let $a,b,d,e\in\Z_{\ge 0}$, $c,f\in\{0,1\}$, and let
\[
v=\alpha^a_1\alpha^b_2\alpha^c_3\gamma_1^{2d+f}\gamma_2^{2e+f}.
\]
Then, by the table above, we have
\[
\LM(v)=x_{11}^{s_{11}}x_{12}^{s_{12}}x_{13}^{s_{13}}
x_{22}^{s_{22}}x_{23}^{s_{23}}x_{33}^{s_{33}}
\]
where 
\begin{align*}
s_{11}&=a+2b+3c+2d+4e+3f,\\
s_{12}&=a+2b+3c+2e+f,\\
s_{13}&=a,\\
s_{22}&=2d+2e+2f,\\
s_{23}&=2b+3c+4e+2f,\\
s_{33}&=2d+f.
\end{align*}
From these relations, we have
\begin{equation*}
c=s_{23} \bmod 2,\quad
f=s_{33} \bmod 2,
\end{equation*}
and
\begin{equation*}
a=s_{13},\quad
b=\frac{s_{23}-2s_{22}+2s_{33}-3c}2,\quad
d=\frac{s_{33}-f}2,\quad
e=\frac{s_{22}-s_{33}-f}2,
\end{equation*}
which imply that $\LM(v)$ uniquely determines $a,b,c,d,e$ and $f$.
\end{proof}

\subsection{Highest weight vectors in $S^3(S^m(\C^n)$ and in $\Lambda^3(S^m(\C^n)$}

Let us denote by $\calq^{\sym3}_{m,D}$ (resp. $\calq^{\sgn}_{m,D}$) the $\rho^D_n$-isotypic component
in $\calq^{\sym3}_m\cong S^3(S^m(\C^n))$ (resp. $\calq^{\sgn}_m\cong\Lambda^3(S^m(\C^n))$).
The nonzero vectors in $(\calq^{\sym3}_{m,D})^{U_n}$
(resp. $(\calq^{\sgn}_{m,D})^{U_n}$) are precisely the the $\Gln$-highest weight vectors in 
$\calq^{\sym3}_{m,D}\cong S^3(S^m(\C^n))$
(resp. $ \calq^{\sgn}_{m,D}\cong \Lambda^3(S^m(\C^n))$).
Hence, the multiplicity of $\rho^D_n$ in $S^3(S^m(\C^n))$ and in $\Lambda^3(S^m(\C^n))$ are given by
\[
\dim\Hom_\Gln(\rho^D_n,S^3(S^m(\C^n)))=\dim(\calq^{\sym3}_{m,D})^{U_n}
\]
and
\[
\dim\Hom_\Gln(\rho^D_n,\Lambda^3(S^m(\C^n)))=\dim(\calq^{\sgn}_{m,D})^{U_n}.
\]

\medskip
We now use the main theorem to obtain a basis for $(\calq^{\sym3}_{m,D})^{U_n}$
and a basis for $(\calq^{\sgn}_{m,D})^{U_n}$.
Recall from equation \eqref{eq:grwt} that if $f\in \calq_m$ and $t.f=\psi^\lambda(t)f$,
then we write $\Gr(f)=m$ and $\wt(f)=\lambda$.
We also denote the cardinality of a finite set $C$ by $\#(C)$.

\begin{cor} 
\begin{enumerate}[(i)]
\item
The set
\[
\Binv(m,D)=\set{f\in\Binv}{\Gr(f)=m,\ \wt(f)=D}
\]
is a basis for $(\calq^{\sym3}_{m,D})^{U_n}$.
Consequently, 
\[
\dim\Hom_\Gln(\rho^D_n,S^3(S^m(\C^n)))=\#(\Binv(m,D)).
\]

\item
The set 
\[
\Bsgn(m,D)=\set{f\in\Bsgn}{\Gr(f)=m,\ \wt(f)=D}
\]
is a basis for $(\calq^{\sym3}_{m,D})^{U_n}$.
Consequently, 
\[
\dim\Hom_\Gln(\rho^D_n,S^3(S^m(\C^n)))=\#(\Bsgn(m,D)).
\eqed
\]
\end{enumerate}
\end{cor}

\section{Examples}

We now compute several examples.
For a set of special elements $f$, we specify $\Gr(f)$ and $\wt(f)$ in the table below:
\begin{equation*}
\begin{array}{|c|c|c|}\hline
f &\Gr(f)&\wt(f)\\ \hline
\alpha_1&1&(3)\\ \hline
\gamma_1&1&(1,1,1)\\ \hline
\alpha_2&2&(4,2)\\ \hline
\gamma_2&2&(3,3)\\ \hline
\alpha_3&3&(6,3)\\ \hline
\gamma^2_1&2&(2,2,2)\\ \hline
\gamma_1\gamma_2&3&(4,4,1)\\ \hline
\gamma^2_2&4&(6,6)\\ \hline
\end{array}
\end{equation*}
If $f=\alpha^a_1\alpha^b_2\alpha^c_3\gamma_1^{2d+f} \gamma_2^{2e+f}$, then 
\begin{align*}
\Gr(f)&=a+2b+3c+2d+4e+3f, \\
\wt(f)&=(3a+4b+6c+2d+6e+4f, 2b+3c+2d+6e+4f, 2d+f).
\end{align*}

\begin{example}[Highest weight vectors in $S^3(S^1(\C^n))$]
The only element $f$ of $\Binv$ such that $\Gr(f)=1$ is $f=\alpha_1$.
Hence, $S^3(S^1(\C^n))$ is irreducible and $S^3(S^1(\C^n))\cong\rho^{(3)}_n$.
\end{example}

\begin{example}[Highest weight vectors in $\Lambda^3(S^1(\C^n))$]
The only element $f$ of $\Bsgn$ such that $\Gr(f)=1$ is $f=\gamma_1$.
Hence, $\Lambda^3(S^1(\C^n))$ is irreducible and $\Lambda^3(S^1(\C^n))\cong\rho^{(1,1,1)}_n$.
\end{example}

\begin{example}[Highest weight vectors in $S^3(S^2(\C^n))$]
The following table lists all the elements $f$ of $\Binv$ with $\Gr(f)=2$. 
\begin{equation*}
\begin{array}{|c|c|c|}\hline
f &\Gr(f)&\wt(f)\\ \hline
\alpha^2_1&1+1=2&(3)+(3)=(6)\\ \hline
\alpha_2&2&(4,2)\\ \hline
\gamma^2_1&2&(2,2,2)\\ \hline
\end{array}
\end{equation*}
Thus, $S^3(S^2(\C^n))$ has the following decomposition:
\[S^3(S^2(\C^n))\cong\rho^{(6)}_n\oplus \rho^{(4,2)}_n\oplus\rho^{(2,2,2)}_n.\]
\end{example}

\begin{example}[Highest weight vectors in $\Lambda^3(S^2(\C^n))$]
The following table lists  all the elements $f$ of $\bbf_{\sgn}$ with $\Gr(f)=2$. 
\begin{equation*}
\begin{array}{|c|c|c|}\hline
f &\Gr(f)&\wt(f)\\ \hline
\alpha_1\gamma_1&1+1=2&(3)+(1,1,1)=(4,1,1)\\ \hline
\gamma_2&2&(3,3)\\ \hline
\end{array}
\end{equation*}
Thus, $\Lambda^3(S^2(\C^n))$ has the following decomposition:
\[\Lambda^3(S^2(\C^n))\cong\rho^{(4,1,1)}_n\oplus \rho^{(3,3)}_n.\]
\end{example}

\begin{example}[Highest weight vectors in $S^3(S^3(\C^n))$]
The following table lists all the elements $f$ of $\bbf_{\mathrm{inv}}$ with $\Gr(f)=3$. 
\begin{equation*}
\begin{array}{|c|c|c|}\hline
f &\Gr(f)&\wt(f)\\ \hline
\alpha^3_1&1+1+1=3&(3)+(3)+(3)=(9)\\ \hline
\alpha_1\alpha_2&1+2=3&(3)+(4,2)=(7,2)\\ \hline
\alpha_3&3&(6,3)\\ \hline
\alpha_1\gamma_1^2&1+2=3&(3)+(2,2,2)=(5,2,2)\\ \hline
\gamma_1\gamma_2&3&(4,4,1)\\ \hline
\end{array}
\end{equation*}
Thus, $S^3(S^3(\C^n))$ has the following decomposition:
\[S^3(S^3(\C^n))=\rho^{(9)}_n\oplus \rho^{(7,2)}_n\oplus\rho^{(6,3)}_n\oplus \rho^{(5,2,2)}_n\oplus \rho^{(4,4,1)}_n.\]
\end{example}

\begin{example}[Highest weight vectors in $\Lambda^3(S^3(\C^n))$]
The following table lists all the elements $f$ of $\bbf_{\sgn}$ with $\Gr(f)=3$. 
\begin{equation*}
\begin{array}{|c|c|c|}\hline
f &\Gr(f)&\wt(f)\\ \hline
\alpha^2_1\gamma_1&2(1)+1=3&2(3)+(1,1,1)=(7,1,1)\\ \hline
\alpha_1\gamma_2&1+2=3&(3)+(3,3)=(6,3)\\ \hline
\alpha_2\gamma_1&2+1=3&(4,2)+(1,1,1)=(5,3,1)\\ \hline
\gamma^3_1&3(1)=3&3(1,1,1)=(3,3,3)\\ \hline
\end{array}
\end{equation*}
Thus, $\Lambda^3(S^3(\C^n))$ has the following decomposition:
\[\Lambda^3(S^3(\C^n))=\rho^{(7,1,1)}_n\oplus \rho^{(6,3)}_n\oplus\rho^{(5,3,1)}_n\oplus \rho^{(3,3,3)}_n.\]
\end{example}

\begin{example}[Highest weight vectors in $S^3(S^4(\C^n))$]
The following table lists all the elements $f$ of $\bbf_{\mathrm{inv}}$ with $\Gr(f)=4$. 
\begin{equation*}
\begin{array}{|c|c|c|}\hline
f &\Gr(f)&\wt(f)\\ \hline
\alpha^4_1&1+1+1+1=4&4(3)=(12)\\ \hline
\alpha^2_1\alpha_2&2+2=4&(6)+(4,2)=(10,2)\\ \hline
\alpha_1\alpha_3&1+3=4&(3)+(6,3)=(9,3)\\ \hline
\alpha^2_1\gamma_1^2&2+2=4&(6)+(2,2,2)=(8,2,2)\\ \hline
\alpha_1(\gamma_1\gamma_2)&1+3=4&(3)+(4,4,1)=(7,4,1)\\ \hline
\alpha^2_2&2+2=4&2(4,2)=(8,4)\\ \hline
\alpha_2\gamma^2_1&2+2=4&(4,2)+(2,2,2)=(6,4,2)\\ \hline
(\gamma^2_1)^2&2(2)=4&2(2,2,2)=(4,4,4)\\ \hline
\gamma^2_2&4&(6,6)\\ \hline
\end{array}
\end{equation*} 
Thus, $S^3(S^4(\C^n))$ has the following decomposition:
\begin{multline*}
S^3(S^4(\C^n))
=\rho^{(12)}_n\oplus \rho^{(10,2)}_n\oplus\rho^{(9,3)}_n\oplus \rho^{(8,4)}_n\oplus \rho^{(8,2,2)}_n\\
{}\oplus\rho^{(7,4,1)}_n\oplus\rho^{(6,6)}_n \oplus \rho^{(6,4,2)}_n\oplus \rho^{(4,4,4)}_n.
\end{multline*}
\end{example}

\begin{example}[Highest weight vectors in $\Lambda^3(S^4(\C^n))$]
The following table lists all the elements $f$ of $\bbf_{\sgn}$ with $\Gr(f)=4$. 
\begin{equation*}
\begin{array}{|c|c|c|}\hline
f &\Gr(f)&\wt(f)\\ \hline
\alpha^3_1\gamma_1&3(1)+1=4&3(3)+(1,1,1)=(10,1,1)\\ \hline
\alpha^2_1\gamma_2&2+2=4&2(3)+(3,3)=(9,3)\\ \hline
\alpha_1\gamma^3_1&1+3(1)=4&(3)+3(1,1,1)=(6,3,3)\\ \hline
\alpha^2_1\gamma_2&2+2=4&(4,2)+(3,3)=(7,5)\\ \hline
\alpha_1\alpha_2\gamma_1&1+2+1=4&(3)+(4,2)+(1,1,1)=(8,3,1)\\ \hline
\alpha_3\gamma_1&2+2=4&(6,3)+(1,1,1)=(7,4,1)\\ \hline
\gamma^2_1\gamma_2&2(1)+2=4&2(1,1,1)+(3,3)=(5,5,2)\\ \hline
\end{array}
\end{equation*} 
Thus, $\Lambda^3(S^4(\C^n))$ has the following decomposition:
\begin{equation*}
\Lambda^3(S^4(\C^n))
=\rho^{(10,1,1)}_n\oplus \rho^{(9,3)}_n\oplus\rho^{(8,3,1)}_n\oplus \rho^{(7,5)}_n
\oplus \rho^{(7,4,1)}_n\oplus\rho^{(6,3,3)}_n\oplus\rho^{(5,5,2)}_n.
\end{equation*}
\end{example}

\begin{example}[Highest weight vectors in $S^3(S^5(\C^n))$]
The following table lists all the elements $f$ of $\bbf_{\mathrm{inv}}$ with $\Gr(f)=5$. 
\begin{equation*}
\begin{array}{|c|c|c|}\hline
f &\Gr(f)&\wt(f)\\ \hline
\alpha^5_1&5(1)=5&5(3)=(15)\\ \hline
\alpha^3_1\alpha_2&3+2=5&3(3)+(4,2)=(13,2)\\ \hline
\alpha_1\alpha^2_2&1+2(2)=5&(3)+2(4,2)=(11,4)\\ \hline
\alpha^2_1\alpha_3&2+3=5&2(3)+(6,3)=(12,3)\\ \hline
\alpha^3_1\gamma_1^2&3(1)+2=5&3(3)+(2,2,2)=(11,2,2)\\ \hline
\alpha_1(\gamma_1^2)^2&1+2(2)=5&(3)+2(2,2,2)=(7,4,4)\\ \hline
\alpha^2_1(\gamma_1\gamma_2)&2(1)+3=5&2(3)+(4,4,1)=(10,4,1)\\ \hline
\alpha_1\gamma^2_2&1+4=5&(3)+(6,6)=(9,6)\\ \hline
\alpha_1\alpha_2\gamma^2_1&1+2+2&(3)+(4,2)+(2,2,2)=(9,4,2)\\ \hline
\alpha_2\alpha_3&2+3=5&(4,2)+(6,3)=(10,5)\\ \hline
\alpha_2(\gamma_1\gamma_2)&2+3=5&(4,2)+(4,4,1)=(8,6,1)\\ \hline
\alpha_3\gamma^2_1&3+2=5&(6,3)+(2,2,2)=(8,5,2)\\ \hline
\gamma^2_1(\gamma_1\gamma_2)&2+3=5&(2,2,2)+(4,4,1)=(6,6,3)\\ \hline
\end{array}
\end{equation*}
Thus, $S^3(S^5(\C^n))$ has the following decomposition:
\begin{multline*}
S^3(S^5(\C^n))
=\rho^{(15)}_n\oplus \rho^{(13,2)}_n\oplus \rho^{(12,3)}_n\oplus\rho^{(11,4)}_n\oplus\rho^{(11,2,2)}_n
\oplus \rho^{(10,5)}_n\oplus\rho^{(10,4,1)}_n\\
{}\oplus \rho^{(9,6)}_n\oplus\rho^{(9,4,2)}_n\oplus\rho^{(8,6,1)}_n\oplus \rho^{(8,5,2)}_n\oplus \rho^{(7,4,4)}_n\oplus\rho^{(6,6,3)}_n.
\end{multline*}
\end{example}

\begin{example}[Highest weight vectors in $\Lambda^3(S^5(\C^n))$]
The following table lists all the elements $f$ of $\bbf_{\sgn}$ with $\Gr(f)=5$. 
\begin{equation*}
\begin{array}{|c|c|c|}\hline
f &\Gr(f)&\wt(f)\\ \hline
\alpha^4_1\gamma_1&4(1)+1=5&4(3)+(1,1,1)=(13,1,1)\\ \hline
\alpha^3_1\gamma_2&3+2=5&3(3)+(3,3)=(12,3)\\ \hline
\alpha^2_1\gamma_1^3&2(1)+3(1)=5&2(3)+3(1,1,1)=(9,3,3)\\ \hline
\alpha^2_1\alpha_2\gamma_1&2(1)+2+1=5&2(3)+(4,2)+(1,1,1)=(11,3,1)\\ \hline
\alpha_1\alpha_2\gamma_2&1+2+2=5&(3)+(4,2)+(3,3)=(10,5)\\ \hline
\alpha_1\gamma_1^2\gamma_2&1+2(1)+2=5&(3)+2(1,1,1)+(3,3)=(8,5,2)\\ \hline
\alpha_1\alpha_3\gamma_1&1+3+1=5&(3)+(6,3)+(1,1,1)=(10,4,1)\\ \hline
\alpha^2_2\gamma_1&2(2)+1=5&2(4,2)+(1,1,1)=(9,5,1)\\ \hline
\alpha_2\gamma^3_1&2+3(1)=5&(4,2)+3(1,1,1)=(7,5,3)\\ \hline
\alpha_3\gamma_2&3+2=5&(6,3)+(3,3)=(9,6)\\ \hline
\gamma_1^5&5(1)=5&5(1,1,1)=(5,5,5)\\ \hline
\gamma_1\gamma^2_1&1+2(2)=5&(1,1,1)+2(3,3)=(7,7,1)\\ \hline
\end{array}
\end{equation*}
Thus, $\Lambda^3(S^5(\C^n))$ has the following decomposition:
\begin{multline*}
\Lambda^3(S^5(\C^n))
=\rho^{(13,1,1)}_n\oplus \rho^{(12,3)}_n\oplus \rho^{(11,3,1)}_n\oplus\rho^{(10,5)}_n
\oplus\rho^{(10,4,1)}_n\oplus\rho^{(9,6)}_n\\
{}\oplus\rho^{(9,5,1)}_n\oplus\rho^{(9,3,3)}_n\oplus \rho^{(8,5,2)}_n\oplus\rho^{(7,7,1)}_n\oplus \rho^{(7,5,3)}_n\oplus \rho^{(5,5,5)}_n.
\end{multline*}
\end{example}

\begin{example}[Highest weight vectors in $S^3(S^6(\C^n))$]
The following table lists all the elements $f$ of $\bbf_{\mathrm{inv}}$ with $\Gr(f)=6$. 
\begin{equation*}
\begin{array}{|c|c|c|}\hline
f &\Gr(f)&\wt(f)\\ \hline
\alpha^6_1&6(1)=6&6(3)=(18)\\ \hline
\alpha^4_1\alpha_2&4(1)+2=6&4(3)+(4,2)=(16,2)\\ \hline
\alpha^2_1\alpha^2_2&2(1)+2(2)=6&2(3)+2(4,2)=(14,4)\\ \hline
\alpha^3_1\alpha_3&3(1)+3=6&3(3)+(6,3)=(15,3)\\ \hline
\alpha^4_1\gamma_1^2&4(1)+2=6&4(3)+(2,2,2)=(14,2,2)\\ \hline
\alpha^2_1(\gamma_1^2)^2&2(1)+2(2)=6&2(3)+2(2,2,2)=(10,4,4)\\ \hline
\alpha^3_1(\gamma_1\gamma_2)&3(1)+3=6&3(3)+(4,4,1)=(13,4,1)\\ \hline
\alpha^2_1\gamma^2_2&2(1)+4=6&2(3)+(6,6)=(12,6)\\ \hline
\alpha^2_1\alpha_2\gamma^2_1&2(1)+2+2=6&2(3)+(4,2)+(2,2,2)=(12,4,2)\\ \hline
\alpha_1\alpha_2\alpha_3&1+2+3=6&(3)+(4,2)+(6,3)=(13,5)\\ \hline
\alpha_1\alpha_2(\gamma_1\gamma_2)&1+2+3=6&(3)+(4,2)+(4,4,1)=(11,6,1)\\ \hline
\alpha_1\alpha_3\gamma^2_1&1+3+2=6&(3)+(6,3)+(2,2,2)=(11,5,2)\\ \hline
\alpha_1\gamma^2_1(\gamma_1\gamma_2)&1+2+3=6&(3)+(2,2,2)+(4,4,1)=(9,6,3)\\ \hline
\alpha^3_2&3(2)=6&3(4,2)=(12,6)\\ \hline
\alpha^2_2\gamma^2_1&2(2)+2=6&2(4,2)+(2,2,2)=(10,6,2)\\ \hline
\alpha_2(\gamma^2_1)^2&2+2(2)=6&(4,2)+2(2,2,2)=(8,6,4)\\ \hline
\alpha_2\gamma^2_2&2(1)+4=6&(4,2)+(6,6)=(10,8)\\ \hline
\alpha_3(\gamma_1\gamma_2)&3+3=6&(6,3)+(4,4,1)=(10,7,1)\\ \hline
(\gamma_1^2)^3&3(2)=6&3(2,2,2)=(6,6,6)\\ \hline
\gamma^2_1\gamma^2_2&2+4=6&(2,2,2)+(6,6)=(8,8,2)\\ \hline
\end{array}
\end{equation*}
Thus, $S^3(S^6(\C^n))$ has the following decomposition:
\begin{multline*}
S^3(S^6(\C^n))
=\rho_n^{(18)}\oplus\rho_n^{(16, 2)}\oplus \rho_n^{(15, 3)}\oplus \rho_n^{(14, 4)}\oplus \rho_n^{(14, 2, 2)}\oplus \rho_n^{(13, 5)}\oplus \rho_n^{(13, 4, 1)}
\\
\oplus 2\rho_n^{(12, 6)}\oplus \rho_n^{(12, 4, 2)}\oplus \rho_n^{(11, 6, 1)}\oplus \rho_n^{(11, 5, 2)}\oplus \rho_n^{(10, 8)} \oplus \rho_n^{(10, 7, 1)}
\\
\oplus \rho_n^{(10, 6, 2)} \oplus \rho_n^{(10, 4, 4)} \oplus \rho_n^{(9, 6, 3)}\oplus \rho_n^{(8, 8, 2)}\oplus \rho_n^{(8, 6, 4)} \oplus\rho_n^{(6, 6, 6)}.
\end{multline*}
\end{example}

\begin{example}[Highest weight vectors in $\Lambda^3(S^6(\C^n))$]
The following table lists all the elements $f$ of $\Bsgn$ with $\Gr(f)=6$. 
\begin{equation*}
\begin{array}{|c|c|c|}\hline
f &\Gr(f)&\wt(f)\\ \hline
\alpha_1^5\gamma_1&5+1=6&5(3)+(1,1,1)=(16,1,1)\\ \hline
\alpha_1^3\alpha_2\gamma_1&3+2+1=6&3(3)+(4,2)+(1,1,1)=(14,3,1)\\ \hline
\alpha_1^2\alpha_3\gamma_1&2+3+1=6&2(3)+(6,3)+(1,1,1)=(13,4,1)\\ \hline
\alpha_1\alpha_2^2\gamma_1&1+2(2)+1=6&(3)+2(4,2)+(1,1,1)=(12,5,1)\\ \hline
\alpha_2\alpha_3\gamma_1&2+3+1=6&(4,2)+(6,3)+(1,1,1)=(11,6,1)\\ \hline
\alpha_1\gamma_1\gamma_2^2&1+1+4=6&(3)+(1,1,1)+(6,6)=(10,7,1)\\ \hline
\alpha_1^3\gamma_1^3&3+3=6&3(3)+3(1,1,1)=(12,3,3)\\ \hline
\alpha_1\alpha_2\gamma_1^3&1+2+3=6&(3)+(4,2)+3(1,1,1)=(10,5,3)\\ \hline
\alpha_3\gamma_1^3&3+3=6&(6,3)+3(1,1,1)=(9,6,3)\\ \hline
\alpha_1\gamma_1^5&1+5=6&(3)+5(1,1,1)=(8,5,5)\\ \hline
\alpha_1^4\gamma_2&4+2=6&4(3)+(3,3)=(15,3)\\ \hline
\alpha_1^2\alpha_2\gamma_2&2+2+2=6&(6)+(4,2)+(3,3)=(13,5)\\ \hline
\alpha_1\alpha_3\gamma_2&1+3+2=6&(3)+(6,3)+(3,3)=(12,6)\\ \hline
\alpha_2^2\gamma_2&2+2+2=6&2(4,2)+(3,3)=(11,7)\\ \hline
\alpha_1^2\gamma_1^2\gamma_2&2+2+2=6&2(3)+2(1,1,1)+(3,3)=(11,5,2)\\ \hline
\alpha_2\gamma_1^2\gamma_2&2+2+2=6&(4,2)+2(1,1,1)+(3,3)=(9,7,2)\\ \hline
\gamma_1^4\gamma_2&4+2=6&4(1,1,1)+(3,3)=(7,7,4)\\ \hline
\gamma_2^3&3(2)=6&3(3,3)=(9,9)\\ \hline
\end{array}
\end{equation*}
Thus, $\Lambda^3(S^6(\C^n))$ has the following decomposition:
\begin{multline*}
\Lambda^3(S^6(\C^n))
= \rho_n^{(16, 1, 1)}\oplus \rho_n^{(15, 3)}\oplus \rho_n^{(14, 3, 1)}\oplus \rho_n^{(13, 5)}\oplus \rho_n^{(13, 4, 1)}
\\
\oplus \rho_n^{(12, 6)}\oplus \rho_n^{(12, 5, 1)}\oplus \rho_n^{(12, 3, 3)}\oplus \rho_n^{(11, 7)}\oplus \rho_n^{(11, 6, 1)}\oplus \rho_n^{(11, 5, 2)}
\\
\oplus \rho_n^{(10, 7, 1)}\oplus \rho_n^{(10, 5, 3)}\oplus \rho_n^{(9, 9)} \oplus \rho_n^{(9, 7, 2)} \oplus \rho_n^{(9, 6, 3)} \oplus \rho_n^{(8, 5, 5)} \oplus\rho_n^{(7, 7, 4)}.
\end{multline*}
\end{example}

\begin{remark}
For a Young diagram $D=(\lambda_1,\lambda_2,\lambda_3)$ with $|D|=3m$,
the multiplicity of $\rho_n^D$ in $S^3(S^m(\C^n))$ is equal to the number of tuples
$(a,b,d,e,c,f)\in\Z_{\ge0}^4\times\{0,1\}^2$ satisfying
\begin{equation*}
\wt(\alpha^a_1\alpha^b_2\alpha^c_3\gamma_1^{2d+f}\gamma_2^{2e+f})=D,
\end{equation*}
which is equivalent to the equations
\begin{equation}\label{equation for multiplicity}
\begin{split}
3a+4b+6c+2d+6e+4f&=\lambda_1, \\
2b+3c+2d+6e+4f&=\lambda_2, \\
2d+f&=\lambda_3.
\end{split}
\end{equation}
The solutions of \eqref{equation for multiplicity} lying in $\Z_{\ge0}^4\times\{0,1\}^2$ are explicitly given by
\begin{gather*}
a=\lambda_1-2k,\quad
b=3k-\lambda_1-2\lambda_2+3\floor{\frac{\lambda_2}2},\quad
d=\floor{\frac{\lambda_3}2},\\
e=\frac{\lambda_1+\lambda_2-2\lambda_3}3-k+\floor{\frac{\lambda_3}2},\quad
c=\lambda_2 \bmod 2,\quad
f=\lambda_3 \bmod 2
\end{gather*}
for $k\in\Z$ whenever $a,b,e\ge0$, where $\floor x$ is the largest integer not exceeding $x$.
Hence we have
\begin{align*}
&\eqspace\dim\Hom_\Gln(\rho^D_n,S^3(S^m(\C^n))) \\
&{}=\#\set{k\in\Z}{
\frac{\lambda_1+2\lambda_2}3-\floor{\frac{\lambda_2}2} \le k \le \min\biggl\{
\floor{\frac{\lambda_1}2},\;\frac{\lambda_1+\lambda_2-2\lambda_3}3+\floor{\frac{\lambda_3}2}
\biggr\}} \\
&{}=\floor{\frac{\min\{\lambda_1-\lambda_2,\lambda_2-\lambda_3\}}6+\frac{2\lambda_1+\lambda_2}6}
+\floor{\frac{\lambda_2}2}+\floor{-\frac{\lambda_1+2\lambda_2}3}+1.
\end{align*}
Similarly, we also have
\begin{align*}
&\eqspace\dim\Hom_\Gln(\rho^D_n,\Lambda^3(S^m(\C^n))) \\
&{}=\floor{\frac{\min\{\lambda_1-\lambda_2,\lambda_2-\lambda_3\}}6+\frac{2\lambda_1+\lambda_2}6+\frac12}
+\floor{\frac{\lambda_2+1}2}+\floor{-\frac{\lambda_1+2\lambda_2}3}.
\end{align*}
We note that both of these multiplicities are of the form
\begin{equation*}
\floor{\frac{\min\{\lambda_1-\lambda_2,\lambda_2-\lambda_3\}}6}+\ve,\quad
\ve\in\{-1,0,1\}.
\end{equation*}
\end{remark}

\end{document}